\documentclass[a4paper,11pt]{amsart}

\usepackage[utf8]{inputenc}
\usepackage[english]{babel}
\usepackage{mathrsfs}
\usepackage{a4wide}
\usepackage{amsmath, amssymb}
\usepackage[pdfborder={1 1 0.3},bookmarks=false]{hyperref}

\newtheorem{theorem}{Theorem}[section]

\newtheorem{lemma}[theorem]{Lemma}
\newtheorem{proposition}[theorem]{Proposition}
\newtheorem{corollary}[theorem]{Corollary}

\theoremstyle{definition}

\usepackage{etoolbox}
\makeatletter
\patchcmd{\@endtheorem}{\@endpefalse}{}{}{}
\patchcmd{\endproof}{\@endpefalse}{}{}{}
\makeatother

\DeclareMathOperator{\Hom}{Hom}

\DeclareMathOperator{\Gal}{Gal}
\DeclareMathOperator{\Spec}{Spec}

\DeclareMathOperator{\Ext}{Ext}

\newcommand\Z{\ensuremath{\mathbf{Z}}}

\newcommand\Q{\ensuremath{\mathbf{Q}}}
\newcommand\F{\ensuremath{\mathbf{F}}}
\newcommand\G{\ensuremath{\mathbb{G}}}
\newcommand{\Alg}[1]{\overline{#1}}
\newcommand{\bad}{\pi}

\newcommand{\gsc}{J}

\newcommand{\spacedot}{\, \cdot \,}
\DeclareMathOperator{\Class}{Cl}

\begin{document}
\title{Extensions of group schemes of $\mu$-type by a constant group scheme}
\author{Hendrik Verhoek}
\begin{abstract}
For a number field $K$, a finite set of primes $S$ not 
containing a fixed prime $p$,
we explain when extensions of group schemes of $\mu_p$ by $\Z/p\Z$ split
over the ring of $S$-integers $O_S$ of $K$.
\end{abstract}
\maketitle

\tableofcontents

\section{Introduction} \label{sec:intro}

Let $p$ be a rational prime and $K$ a number field.
Let $S$ be a finite set of primes in $K$ that does not contain primes above $p$.
Let $\pi$ be a prime ideal above $p$ in $O_S$ and 
let $\widehat{O_S}$ be the completion of $O_S$ at $\pi$.
Denote by $\Ext^1_{O_S}(\mu_p,\Z/p\Z)$ the group of equivalence classes of extensions of $\mu_p$ by the constant
group scheme $\Z/p\Z$ in the category of finite flat commutative group schemes over $O_S$.
Our main goal is to calculate the group $\Ext^1_{O_S}(\mu_p,\Z/p\Z)$:

\begin{theorem} \label{thm:main}
Suppose $p$ does not split in $K/\Q$.
Let $\widehat{L} = \widehat{O_S}[\zeta_p/p]$
and $\omega : \Gal(\Q(\zeta_p)/\Q) \rightarrow \F_p^*$ be the cyclotomic character at $p$.
Suppose that the ${\omega^2}$-eigenspace of the
$p$-torsion of the class group of $O_S[\zeta_p/p]$ is trivial. Then
$$
\Ext^1_{O_S}(\mu_p,\Z/p\Z) \simeq_{\F_p}
\ker \left( (O_S[\zeta_p/p]^*/(O_S[\zeta_p/p]^{*})^p)_{\omega^2}
\longrightarrow (\widehat{L}^*/(\widehat{L}^{*})^p)_{\omega^2} \right) .
$$
\end{theorem}

Finite flat commutative group schemes of $p$-power order
over a base where $p$ is invertible,
are \'etale group schemes and therefore just Galois modules.
Therefore we will consider in Section \ref{sec:modules}
extensions of modules with a group action.
In Section \ref{sec:groupschemes}
we move on to extensions of the finite flat group schemes
$\Z/p\Z$ by $\mu_p$ that are killed by $p$ and prove Theorem \ref{thm:main}.
Finally, we calculate for various $K$ and $S$ the group $\Ext^1_{O_S}(\mu_p,\Z/p\Z)$
using Theorem \ref{thm:main}.

\section{Extensions of modules} \label{sec:modules}

Let $R$ be a commutative unitary ring such that $p\cdot R=0$ 
and let $G$ be a group.
When we say $R$-module, we mean a left $R$-module.
We will consider extensions of $R$-modules with an action of $G$, as
a preparation for the next section,
where we will discuss extensions of finite flat group schemes.
We will use the following theorem of Grothendieck:

\begin{theorem} \label{thm:spectral_low_degree}
Let $C_1,C_2$ and $C_3$ be abelian categories, such that $C_1$ and $C_2$ have enough injectives.
Let $F_1 : C_1 \rightarrow C_2$ be
a left exact functor that maps injective objects in $C_1$
to acyclic objects in $C_2$ and let $F_2 : C_2 \rightarrow C_3$ be a left exact functor.
Then there is an exact sequence of low degree terms:
\begin{align*}
0 \longrightarrow & (R^1 F_2)(F_1(A)) \longrightarrow (R^1 (F_2 F_1))(A) \longrightarrow \\
& F_2((R^1 F_1)(A)) \longrightarrow 
(R^2 F_2)(F_1(A)) \longrightarrow  (R^2 (F_2 F_1))(A)
\end{align*}
\end{theorem}
\begin{proof}
See \cite[Theorem 5.8.3, p. 151]{Weibel:1994}.
\end{proof}

Let $A$ and $B$ be two $R[G]$-modules such that $G$ acts trivially on $B$ and such
that $H$ acts trivially on $A$.
Let $\chi : G \rightarrow (\Z/p\Z)^*$ and suppose that $H$ is contained in $\ker(\chi)$.
Denote by $B(\chi)$ the $G$-module that has underlying group structure the one of 
$B$ and where the $G$-action is given by $\sigma b := \chi(\sigma) b$ for all $\sigma \in G$ and all $b \in B$.
We let $G$ act on $\Hom_{H}(B,A)$ through
the action of $G$ on $A$.
Denote by $\Hom_{H}(B,A)_{\chi}$ the subgroup of 
$\Hom_{H}(B,A)$ on which $\Gamma = G/H$ acts through $\chi$.

\begin{lemma} \label{lem:hom_twist}
We have the following isomorphisms of groups:
$$
\Hom_{G}(B(\chi),A)
\simeq 
\Hom_{H}(B,A)_{\chi} .
$$
\end{lemma}
\begin{proof}
Let $\psi : B \rightarrow B(\chi)$ be an $H$-linear isomorphism
and let 
$\phi : \Hom_{G}(B(\chi),A) \rightarrow \Hom_{H}(B,A)_{\chi}$
such that $f \mapsto f \circ \psi$.
The morphism $f \circ \psi$ is indeed an element in $\Hom_{H}(B,A)_{\chi}$
because for all $b$ in $B$ and $\sigma$ in $G$ the following equalities hold:
\begin{align*}
(\sigma (f \circ \psi))(b) = & \sigma((f \circ \psi)(b)) \\
= & f (\sigma (\psi(b))) = f (\chi(\sigma) \psi(b)) \\
= & \chi(\sigma) f (\psi(b)) .
\end{align*}
Note that the second equality follows because $f$ is $G$-linear.
Next we prove that the inverse morphism of $\phi$ is just precomposing with $\psi^{-1}$.
For $g \in \Hom_{H}(B,A)_{\chi}$
we have $g \circ \psi^{-1} \in \Hom_{G}(B(\chi),A)$, because
for all $b$ in $B(\chi)$ and all $\sigma$ in $G$:
\begin{align*}
(\sigma (g \circ \psi^{-1}))(b) = & \chi(\sigma) ((g \circ \psi^{-1})(b)) \\
= & (g \circ \psi^{-1})( \chi(\sigma) b) \\
= & (g \circ \psi^{-1})( \sigma \cdot b) . 
\end{align*} 
\end{proof}

\begin{proposition} \label{prop:grpext_twist}
Let $A$ and $B$ be two $R[G]$-modules such that $G$
acts trivial on $B$ and such that $H$ acts trivial on $A$.
Then
$$
\Ext_{H}^{1}(B,A)_{\chi}  \simeq \Ext_{G}^{1}(B(\chi),A) 
$$
as $R$-modules.
\end{proposition}
\begin{proof}
We consider the following two functors:
The left exact functor $F_1( \spacedot )=\Hom_H( \spacedot \, ,A)$
from $R[G]$-modules to $R[\Gamma]$-modules
and the exact functor $F_2$ ``taking $\chi$-eigenspaces''
from the category of $R[\Gamma]$-modules to
$R[\Gamma]$-modules.
With these two functors $F_1$ and $F_2$, we apply Theorem \ref{thm:spectral_low_degree}.
Since $F_2$ is exact, the functors $F_1$ and $F_2$ give rise to the exact sequence
\begin{align*}
0 & \longrightarrow (R^{1}F_2)(\Hom_{H}(B,A)) \longrightarrow 
R^1( F_2 F_1 )(B) \\
& \longrightarrow \Ext^{1}_{H}(B,A)_{\chi} \longrightarrow 
(R^{2} F_2 )(\Hom_{H}(B,A)) \longrightarrow \dots .
\end{align*}
Since $F_2$ is exact, the $R[\Gamma]$-modules
$R^1(F_2 F_1 )(B)$ and $\Ext^{1}_{H}(B,A)_{\chi}$ are isomorphic.
But $(F_2 F_1)(B)$ is isomorphic to $\Hom_G(B(\chi),A)$ by Lemma \ref{lem:hom_twist}.

Define the functor $T_1$  "twisting with $\chi$" from the category of $R[G]$-modules to itself,
the functor $T_2(\spacedot) = \Hom_G(\spacedot \, , A)$ from 
the category of $R[G]$-modules to the category of $R$-modules
and the forgetful functor $F$ that forgets the $\Gamma$ action and goes from 
the category of $R[\Gamma]$-modules to the category of $R$-modules.
Then we have a natural isomorphism $F F_2 F_1 \simeq T_2 T_1$.
The functor $T_1$ sends injective objects to injective objects, which are in particular acyclic
objects. Hence, we can apply Theorem \ref{thm:spectral_low_degree} to get the following exact sequence:
\begin{align*}
0 & \longrightarrow (R^{1}T_2)( B(\chi) ) \longrightarrow 
R^1( T_2 T_1 )(B)  \longrightarrow  T_2 (R^1 T_1(B)) \longrightarrow 
(R^{2} T_2 ) (T_1 (B) ) \longrightarrow \dots \, .
\end{align*}
Since $T_1$ is exact, we obtain that $(R^{1}T_2)( B(\chi) ) = \Ext^1_G(B(\chi),A)$ is isomorphic to $R^1( T_2 T_1 )(B)$
as $R$-modules.
Putting everything together, we now have isomorphisms of $R$-modules 
$$
\Ext^{1}_{H}(B,A)_{\chi} \simeq R^1(F_2 F_1 )(B) \simeq R^1( T_2 T_1 )(B)  \simeq \Ext_{G}^{1}(B(\chi),A) ,
$$
which is what we wanted to show.
\end{proof}

When we take $\chi$ to be the trivial character in Proposition \ref{prop:grpext_twist},
we obtain the $\Gamma$-invariant extensions, which are as expected just
extensions of $R[G]$-modules.
We conclude by remarking that for 
two $R[G]$-modules $A$, $B$ and for a character 
$\chi$ of $G$, the $R$-module $\Ext^{1}_{G}(A,B)$ is isomorphic to
the $R$-module $\Ext^{1}_{G}(A(\chi),B(\chi))$.
Here $A(\chi)$ (resp. $B(\chi)$) is the twist of $A$ (resp. $B$) by $\chi$.

\section{Extensions of group schemes} \label{sec:groupschemes}

Recall that $\pi$ denotes the prime ideal above $p$ in the ring of $S$-integers $O_S$ of the number field $K$.
In this section, the ring $R$ will be either the ring of $S$-integers $O_S$, 
the number field $K$, the completion of $O_S$ with respect to $\bad$
or the fraction field of such a completion of $O_S$.
Since $p$ does not split in $K/\Q$, in each case
we can talk about the fraction field of $R$, which we denote by $F$.
Furthermore, let $L=F(\zeta_p)$ and $\Gamma = \Gal(L/F)$.

First we state some facts from \cite[Section 8.7-8.10]{KatzMazur:1985}.
Let $r$ be a unit in $R$.
Consider the finite flat commutative group scheme 
$$
T(r) = \Spec( \prod_{i=0}^{p-1} R[X_i] /(X_{i}^p - r^{i})  )
= \amalg_{i=0}^{p-1} \,  \Spec( R[X_i ]/(X_i^p - r^{i}) )
$$ over $R$.
The scheme $T(r)$ is an extension of
$\Z/p\Z$ by $\mu_p$.
%
For an $R$-algebra $A$, the $A$-valued points in $T(r)$ are pairs $(a,i/p) \in (A,\Q)$
such that $a^p = r^{i}$ and $0 \leq i \leq p-1$.
The group law of $T(r)$ can be described by
$$
(a,i/p) \times (b,j/p) = 
\left\{
\begin{array}{rl}
  (ab,(i+j)/p) \, ,& i+j < p \\
  (ab/r,(i+j-p)/p)  \, ,& i+j \geq p  \quad . \\
\end{array}
\right.
$$
The group schemes $T(r^p)$ are split extensions of $\Z/p\Z$ by $\mu_p$
and we see that in that case we have: 
$$
(a,i/p) = (ar^{-i},0) \times (r^{i},i/p) .
$$
%
%
If $r$ and $r'$ are units in $R$,
then the group schemes $T(r)$ and $T(r')$ are isomorphic if and only if $r$ and $r'$ generate the same
subgroup in $R^*/(R^{*})^p$.

\begin{lemma} \label{lem:fppf}
The sequence
$$
0 \rightarrow R^*/(R^{*})^p \rightarrow \Ext^{1}_{R,[p]}(\Z/p\Z,\mu_p) \rightarrow \Class(R)[p] \rightarrow 0 
$$
is exact.
\end{lemma}
\begin{proof}
(cf. \cite{Mazur:1977} and \cite[Proposition 2.2]{Schoof:2009}).
Apply $\Hom(\spacedot,\mu_p)$ to 
the exact sequence of fppf sheaves 
$0 \rightarrow \Z \rightarrow \Z \rightarrow \Z/p\Z \rightarrow 0$
to obtain
$$
0 \rightarrow \mu_p(R) \rightarrow \Ext^{1}_{R}(\Z/p\Z,\mu_p) \rightarrow 
\Ext^{1}_{R}(\Z,\mu_p) \simeq H^1_{\text{fppf}}(\Spec(R),\mu_p) \rightarrow 0 .
$$
On the other hand, we apply the global section functor 
to the Kummer sequence of fppf sheaves
$$
0 \longrightarrow \mu_p \longrightarrow \G_m \stackrel{[p]}{\longrightarrow} \G_m \longrightarrow 0
$$
to obtain
$$
0 \rightarrow R^*/(R^{*})^p \rightarrow H^1_{\text{fppf}}(\Spec(R),\mu_p) \rightarrow \Class(R)[p] \rightarrow 0 ,
$$
where $\Class(R)$ is the class group of $R$.
The lemma follows by \cite[Proposition 2.2 i)]{Schoof:2009} that says
that $H^1_{\text{fppf}}(\Spec(R),\mu_p) \simeq \Ext^{1}_{R,[p]}(\Z/p\Z,\mu_p)$.
\end{proof}

We focus again on the group $\Ext^1_{O_S}(\mu_p,\Z/p\Z)$.
If $R$ is a completion of $O_S$ at $\bad$,
the group $\Ext^1_R(\mu_p,\Z/p\Z)$ is trivial since $\mu_p$ is connected
and the connected-\'etale exact sequence gives a section for such extensions.
Therefore, extensions of $\mu_p$ by $\Z/p\Z$ are locally split and hence killed by $p$.
Since the completion of $O_S$ at $\bad$ is flat over $O_S$, extensions of $\mu_p$ by $\Z/p\Z$
over the ring $O_S$ are also killed by $p$ and
$\Ext^1_{O_S}(\mu_p,\Z/p\Z) = \Ext^1_{O_S,[p]}(\mu_p,\Z/p\Z)$.
Let $\omega : \Gamma = \Gal(L/F) \rightarrow \F_p^*$
be the character such that for all $\sigma \in \Gamma$ we have $\sigma(\zeta_p) = \zeta_p^{\omega(\sigma)}$.
The scheme $\mu_{p}$ over $R[\zeta_p/p]$ is a constant group scheme
and $(\mu_{p})_{R[\zeta_p/p] }  \simeq_{\F_p} (\Z/p\Z)_{R[\zeta_p/p]}$.
For integers $0 \leq i,j \leq p-2$ we have the following isomorphisms of $\F_p$-modules:
$$
\Ext_{ R[\zeta_p/p]}^{1}  \left( \Z/p\Z(\omega^i),\Z/p\Z(\omega^j) \right) \simeq_{\F_p} \Ext_{ R[\zeta_p/p] }^{1}(\Z/p\Z,\mu_{p}) .
$$

\begin{lemma} \label{lem:ext_iso_units}
$\Ext^1_{R[1/p],[p]}(\Z/p\Z(\omega^i),\mu_p)  \simeq_{\F_p} \Ext^1_{ R[\zeta_p/p]  ,[p] }(\Z/p\Z,\mu_p)_{\omega^i}$.
\end{lemma}
\begin{proof}
This follows immediately from Proposition \ref{prop:grpext_twist}.
\end{proof}


\begin{corollary}
If $\zeta_p \notin R$, then
$\Ext^1_{ R[\zeta_p/p],[p] }(\Z/p\Z,\mu_p) \simeq_{\F_p} \bigoplus\limits_{i=0}^{p-2} \Ext^1_{ R[1/p] ,[p] }(\Z/p\Z(\omega^i),\mu_p)$.
\end{corollary}
\begin{proof}
The group $\Ext^1_{ R[\zeta_p/p],[p] }(\Z/p\Z,\mu_p)$ is an $\F_p[\Gamma]$-module.
Hence it can be decomposed as
$$
\Ext^1_{ R[\zeta_p/p],[p] }(\Z/p\Z,\mu_p) \simeq_{\F_p[\Gamma]} \oplus_i \Ext^1_{ R[\zeta_p/p],[p] }(\Z/p\Z,\mu_p)_{\omega_i} .
$$
By Lemma \ref{lem:ext_iso_units} each summand is isomorphic to 
$\Ext^1_{ R[1/p] ,[p] }(\Z/p\Z(\omega^i),\mu_p)$ as an $\F_p$-module.
\end{proof}

\begin{lemma}[\cite{Schoof:2003}, Corollary 2.4] \label{prop:mayervietoris}
Let $\gsc'$ and $\gsc''$ be two finite flat commutative group schemes
over $O_S$, let $p$ be a prime and let $\widehat{O_S} = (O_S \otimes \Z_p )$.
Then the following sequence is exact:
\begin{align*}
0 \rightarrow \Hom_{O_S}(\gsc'',\gsc') \rightarrow 
\Hom_{ \widehat{O_S}  }(\gsc'',\gsc') \times
\Hom_{ O_S[1/p] }(\gsc'',\gsc') \rightarrow
\Hom_{  \widehat{O_S}[1/p] }(\gsc'',\gsc') \\ 
\rightarrow \Ext^{1}_{  O_S }(\gsc'',\gsc')
\rightarrow \Ext^1_{ \widehat{O_S} }(\gsc'',\gsc') \times
\Ext^1_{ O_S[1/p] }(\gsc'',\gsc') \rightarrow 
\Ext^1_{ \widehat{O_S}[1/p] }(\gsc'',\gsc')
\end{align*}
\end{lemma}

\begin{lemma} \label{lem:hom_mup_const}
If $p$ does not split in $K/\Q$, then
$\Hom_{ O_S[1/p] }(\mu_p,\Z/p\Z) \simeq \Hom_{\widehat{O_S}[1/p] }(\mu_p,\Z/p\Z)$.
\end{lemma}
\begin{proof}
If $\zeta_p \in K$ then both groups are cyclic of order $p$.
If $\zeta_p \notin K$ then both groups are trivial.
\end{proof}
We are now ready to prove Theorem \ref{thm:main}.

\begin{proof}[Proof of Theorem \ref{thm:main}]
Consider the exact sequence of $\F_p[\Gamma]$-modules of Lemma \ref{lem:fppf}:
$$
0 \rightarrow O_S[\zeta_p/p]^*/(O_S[\zeta_p/p]^{*})^p \rightarrow
\Ext^1_{O_S[\zeta_p/p],[p]}(\Z/p\Z,\mu_p) \rightarrow
\Class(O_S[\zeta_p/p])[p] \rightarrow 0 .
$$
The sequence is still left exact after taking $\omega^2$-eigenspaces.
The condition that the $\omega^2$-eigenspace of the $p$-torsion of
the class group $O_S[\zeta_p/p]$, denoted by $\Class(O_S[\zeta_p/p])[p]_{\omega^2}$, is trivial implies that
$$
\left( O_S[\zeta_p/p]^*/(O_S[\zeta_p/p]^{*})^p \right)_{\omega^2}
\simeq_{\F_p[\Gamma]} \Ext^1_{O_S[\zeta_p/p],[p]}(\Z/p\Z,\mu_p)_{\omega^2} .
$$
Remember that we assume that $p$ does not split in $K/\Q$, hence $\widehat{O_S}[1/p]$ is a field.
We obtain from Lemma \ref{prop:mayervietoris},
together with Lemma \ref{lem:hom_mup_const},
the following exact sequence of $\F_p$-modules:
\begin{align} \label{eq:exseq}
0 \rightarrow & \Ext^1_{O_S,[p]}(\mu_p,\Z/p\Z) \rightarrow \Ext^1_{O_S[1/p],[p]}(\mu_p,\Z/p\Z)
\rightarrow \Ext^1_{\widehat{O_S}[1/p],[p]}(\mu_p,\Z/p\Z) .
\end{align}
Twisting by the character $\omega$ gives the following two isomorphisms:
\begin{align*}
\Ext^1_{O_S[1/p]}\left( \mu_p,\Z/p\Z \right) & \simeq_{\F_p}
\Ext^1_{O_S[1/p]}\left( \Z/p\Z(\omega^2),\mu_p \right) \\
\Ext^1_{\widehat{O_S}[1/p]}(\mu_p,\Z/p\Z) & \simeq_{\F_p}
\Ext^1_{\widehat{O_S}[1/p]}(\Z/p\Z(\omega^2),\mu_p) .
\end{align*}
In particular, we have isomorphisms between the $p$-torsion
subgroups of these extension groups.
From (\ref{eq:exseq}) we obtain
$$
0 \rightarrow \Ext^1_{O_S, [p]}(\mu_p,\Z/p\Z) \rightarrow \Ext^1_{O_S[1/p],[p]}\left( \Z/p\Z(\omega^2),\mu_p \right)
\rightarrow \Ext^1_{\widehat{O_S}[1/p],[p]}\left( \Z/p\Z(\omega^2),\mu_p \right) .
$$
By Lemma \ref{lem:ext_iso_units} we obtain 
$$
0 \longrightarrow \Ext^1_{O_S}(\mu_p,\Z/p\Z) \longrightarrow 
\left( O_S[\zeta_p/p]^*/(O_S[\zeta_p/p]^{*})^ p \right)_{\omega^2} \longrightarrow 
\left( \widehat{L}^*/(\widehat{L}^{*})^p \right)_{\omega^2} .
$$
\end{proof}

\section{Example calculations} \label{sec:example_ext}

We calculate, using the isomorphism of Theorem \ref{thm:main},
for specific $p$ and $O_S$ the extension group
$\Ext^1_{O_S}(\mu_p,\Z/p\Z)$.
We will use the following lemma in the computations:

\begin{lemma} \label{lem:padic_unit_comp}
The group $\Q_2^* / (\Q_2^{*})^2$ is generated
by $2,3$ and $5$.
For $p>2$ the group $\Q_p^* / (\Q_p^{*})^p$ is generated
by $p$ and $1+p$. 
\end{lemma}
\begin{proof}
We have the following isomorphism of groups:
\begin{align} \label{eqn:Qpdecomp}
\Q_p^* \simeq \mu_{p-1} \times p^{\Z} \times (1+p\Z_p) .
\end{align}
First we consider the case $p > 2$.
Then $(\Q_p^{*})^p = \mu_{p-1} \times p^{p\Z} \times (1+p^2\Z_p)$,
where we used Hensel's Lemma to obtain the equality $(1+p\Z_p)^p = (1+p^2\Z_p)$.
The lemma follows from.

Now suppose that $p=2$.
A unit $x \in \Z_{2}^{*}$ is a square
if and only if $x \equiv 1 \pmod{8}$.
The lemma follows again from the isomorphism \eqref{eqn:Qpdecomp}
and the fact that $3$ and $5$ are independent mod $\Q_2^{*2}$:
if they were not independent, $15$ would be a square in $\Q_2$, but $15 \not \equiv 1 \pmod{8}$.
\end{proof}

\subsubsection*{The extension group $\Ext_{\Z[\frac{1}{3}]}^{1}(\mu_{2},\Z/2\Z)$}

We show that $\Ext_{\Z[\frac{1}{3}]}^{1}(\mu_{2},\Z/2\Z)$ is trivial.
Let $K=\Q$, $p=2$ and $S=\{3\}$.
It suffices to show that the homomorphism
\begin{align} \label{eqn:extmor1}
\Z[\frac{1}{6}]^*/\Z[\frac{1}{6}]^{*2} \longrightarrow \Q_2^*/\Q_2^{*2}
\end{align}
is injective.
The non-squares in $\Z[\frac{1}{6}]^*$ are generated by $2,3$ and $-1$.
By Lemma \ref{lem:padic_unit_comp}, the non-squares in $\Q_2^*$
are generated by $2,3$ and $5$.
Hence the homomorphism in \eqref{eqn:extmor1} is injective.

\subsubsection*{The extension group $\Ext^1_{\Z[\frac{1}{2},i]}(\mu_3,\Z/3\Z)$}

Let $K=\Q(i)$, $p=3$ and $S= \{ (1+i) \}$.
Hence $O_S = \Z[i,\frac{1}{2}]$.
The group $\Gamma$ is the Galois group of the extension $\Q(\zeta_{12})/\Q(i)$
and has order $4$.
The cyclotomic character $\omega$ at $3$ is quadratic, so $\omega^2$ is trivial.
The Hilbert class field of $\Q(\zeta_{12})$ is trivial.
We will show that the group $\Ext^1_{\Z[\frac{1}{2},i]}(\mu_3,\Z/3\Z)$ is trivial.
It suffices to show that 
$$
\left( \Z[\zeta_{12},\frac{1}{6}]^*/\Z[\zeta_{12},\frac{1}{6}]^{*3} \right)^{\Gamma}
\longrightarrow \left( \Q_3(\zeta_{12})^*/\Q_3(\zeta_{12})^{*3} \right)^{\Gamma} 
$$
is injective.

Let $F_1$ be the functor from the category of $\Z[G_\Q]$-modules to 
the category of $\Z[\Gamma]$-modules defined by taking $\Gal(\Alg{\Q}/\Q(\zeta_{12}))$-invariants.
The functor $F_1$ sends injective objects to acyclic ones.
Similarly, let $F_2$ be the functor of taking $\Gamma$-invariants
from the category of $\Z[\Gamma]$-modules to 
the category of abelian groups.
We apply Theorem \ref{thm:spectral_low_degree} with the two functors $F_1$ and $F_2$
described above, 
and we take the object $A$ of Theorem \ref{thm:spectral_low_degree} to be the $G_{\Q(i)}$-module $\mu_3$.
Since the order of $\Gamma$ is coprime with the order of $\mu_3$, the derived functors of $F_2$ are zero.
From the long exact sequence of Theorem \ref{thm:spectral_low_degree} we see that
$$
\left( \Z[\zeta_{12},\frac{1}{6}]^*/\Z[\zeta_{12},\frac{1}{6}]^{*3} \right)^{\Gamma} \simeq \Z[\frac{1}{6}]^*/\Z[\frac{1}{6}]^{*3} 
$$
and that
$$
\left( \Q_3(\zeta_{12})^*/\Q_3(\zeta_{12})^{*3} \right)^{\Gamma} \simeq \Q_3^*/\Q_3^{*3} .
$$
We proceed as in the previous example.
%
%

\subsubsection*{The extension group $\Ext^1_{\Z[\frac{1}{7}]}(\mu_2,\Z/2\Z)$}

Let $K=\Q$, $p=2$ and $S= \{ 7 \}$.
Note that $-7$ is a $2$-adic square. Hence the kernel of
$$
\Z[\frac{1}{14}]^*/\Z[\frac{1}{14}]^{*2} \longrightarrow \Q_2^*/\Q_2^{*2} 
$$
is non-trivial and of order $2$. A non-trivial extension of $\mu_2$ by $\Z/2\Z$ over $\Z[\frac{1}{7}]$
is generically isomorphic to the extension $T(-7)$ of $\Z/2\Z$ by $\mu_2$.
However, this extension is locally at $2$ a trivial extension. 
The Hopf algebra of such a non-trivial extension is given by
$$
\Z[\frac{1}{7}][X,Y]/(X^2-X - Y,Y^2+2Y)
$$
with coalgebra maps $\Delta$ (comultiplication), $\epsilon$ (counit) and $S$ (coinverse):
\begin{align*}
\Delta(X)  & = X \otimes 1 + 1 \otimes X - 2 X \otimes X + \frac{1}{7} Y \otimes Y 
 - \frac{2}{7}( Y \otimes XY + XY \otimes Y) + \frac{4}{7} ( XY \otimes XY) \\
\Delta(Y) & = Y \otimes 1 + 1 \otimes Y + Y \otimes Y  \\
\epsilon(X) & = 0 , \quad \epsilon(Y) = 0 \\
S(X) & = -X , \quad S(Y) = Y .
\end{align*}
This group scheme is isomorphic to the $2$-torsion subgroup scheme of the elliptic curve $J_0(49)$.

\bibliography{index}
\bibliographystyle{alpha}
\end{document}